\documentclass[11pt,reqno]{amsart}
\usepackage{amsmath,amssymb,amsthm,mathtools}
\usepackage[margin=1.15in]{geometry}
\usepackage{booktabs}
\usepackage{listings}
\usepackage[colorlinks=true,linkcolor=blue,citecolor=blue,urlcolor=blue]{hyperref}

\theoremstyle{plain}
\newtheorem{theorem}{Theorem}[section]
\newtheorem{lemma}[theorem]{Lemma}
\newtheorem{proposition}[theorem]{Proposition}
\newtheorem{corollary}[theorem]{Corollary}
\theoremstyle{definition}
\newtheorem{question}[theorem]{Question}
\newtheorem{remark}[theorem]{Remark}

\newcommand{\FF}{\mathbb{F}}
\newcommand{\GL}{\mathrm{GL}}

\DeclareMathOperator{\Gam}{\Gamma}

\begin{document}

\title[Maximal connectivity graph-codes in every degree at least four]
{Maximal connectivity graph-codes exist in every degree at least four:\\
$f(d)=2^{d}$ for all $d\ge 4$, answering a question of Alon}

\author{Chenxiao Tian}
\address{Princeton University, Princeton, NJ 08544, USA}
\email{ct3471@alumni.princeton.edu}

\date{\today}


\begin{abstract}
A family $\mathcal{C}$ of spanning subgraphs of a graph $H$ is a \emph{connectivity code} if the symmetric difference of any two distinct members is a connected spanning subgraph of $H$; let $m(H)$ be the largest size of such a family. If $H$ is $d$-regular then $m(H)\le 2^{d}$, and $f(d)$ denotes the largest value attained by $m(H)$ for infinitely many $d$-regular graphs $H$. Alon proved that $f(d)=2^{d}$ for every $d$ exceeding an absolute constant $d_{0}$, by showing that every $d$-regular mild expander attains the bound; he also showed $f(2)<2^{2}$ and $f(3)<2^{3}$, and asked whether $f(d)=2^{d}$ holds for all $d\ge 4$.

We answer this affirmatively: $f(d)=2^{d}$ for every $d\ge 4$. Together with the known values in degrees $2$ and $3$, this determines $f(d)$ for every $d$. The construction is elementary and uniform in $d$, and uses neither expansion nor the local lemma. The base graph is $K_{d,d}$, with edges labelled by the first $d$ iterates of an irreducible operator $T$ on $\FF_2^{d}$; for $d\ge 7$ a first-moment count over a single irreducible conjugacy class in $\GL_{d}(2)$ shows that $T$ may be chosen so that every nonzero linear combination of the labels selects a connected spanning subgraph containing a cycle, and exact certificates cover $4\le d\le 8$. Random cyclic voltage lifts then turn each base example into infinitely many pairwise nonisomorphic $d$-regular examples without decreasing the dimension of the code. Along the way we obtain $m(K_{d,d})=2^{d}$ for every $d\ge 4$.
\end{abstract}

\maketitle

\section{Introduction}

Let $H=(V,E)$ be a finite simple graph. Following Alon \cite{AlonConn}, a family $\mathcal{C}\subseteq 2^{E}$ of spanning subgraphs of $H$ is a \emph{connectivity code} for $H$ if $A\triangle B$ is a connected spanning subgraph of $H$ for all distinct $A,B\in\mathcal{C}$, and $m(H)$ denotes the maximum cardinality of a connectivity code for $H$. This is the natural graph-theoretic analogue of a binary code of large minimum distance: for subgraphs, ``far apart'' is interpreted as ``connected and spanning''. The code is called \emph{linear} if $\mathcal{C}$ is a subspace of $\FF_2^{E}$.

Since no two codewords may agree on the edge set of a nontrivial cut, $m(H)\le 2^{k'(H)}\le 2^{\delta(H)}$, where $k'(H)$ is the edge connectivity and $\delta(H)$ the minimum degree. Equality $m(K_{n})=2^{n-1}$ was established in \cite{AGKMS}, and $m(C_{3}\times C_{3})=m(C_{4}\times C_{4})=2^{4}$ in \cite{BGMW}. On the other hand, equality can fail badly: it is shown in \cite{AlonConn} that $m(K_{t}\times C_{s})\le 2^{t}$ for $s>(2t+1)2^{t-1}$, although $K_{t}\times C_{s}$ is $(t+1)$-regular and $(t+1)$-edge-connected.

For $d\ge 2$ let
\[
f(d)\ =\ \max\bigl\{\,q\ :\ m(H)=q \text{ for infinitely many pairwise nonisomorphic $d$-regular graphs } H \,\bigr\}.
\]
Thus $f(d)\le 2^{d}$ for every $d$. The main theorem of \cite{AlonConn} states that every $d$-regular graph in which each set $W$ of at most half the vertices sends at least $c|W|\log d$ edges to its complement satisfies $m(H)=2^{d}$; since such expanders exist in every sufficiently large degree, there is an absolute constant $d_{0}$ with $f(d)=2^{d}$ for all $d\ge d_{0}$. The proof, which uses the asymmetric Lovász Local Lemma, requires $d$ to be large (in the form given there, $d\ge 1000$), and Alon remarks explicitly that it would not yield small degrees even if the constants were optimised. In the opposite direction $f(2)=2<2^{2}$, and $f(3)=4<2^{3}$ by an application of the Plotkin bound \cite{Plotkin}. This left precisely the range $4\le d<d_{0}$ open, and prompted the following question, which is Question~7.1 of Alon's survey \cite{AlonSurvey}.

\begin{question}[Alon]\label{q:alon}
Is $f(d)=2^{d}$ for all $d\ge 4$?
\end{question}

We answer Question~\ref{q:alon} in the affirmative.

\begin{theorem}\label{thm:main}
For every integer $d\ge 4$ we have $f(d)=2^{d}$. More precisely, for every $d\ge 4$ there are infinitely many pairwise nonisomorphic finite simple $d$-regular bipartite graphs carrying a linear connectivity code of dimension $d$.
\end{theorem}

Combining Theorem~\ref{thm:main} with the two exceptional degrees recalled above, the function $f$ is now completely determined.

\begin{corollary}\label{cor:complete}
$f(2)=2$, $f(3)=4$, and $f(d)=2^{d}$ for every $d\ge 4$.
\end{corollary}

The construction also settles the value of $m$ for the complete bipartite graphs, a finite statement not covered by the expander theorem of \cite{AlonConn} in small degrees.

\begin{corollary}\label{cor:kdd}
$m(K_{d,d})=2^{d}$ for every $d\ge 4$.
\end{corollary}

Our argument is elementary, uniform in $d$, and independent of the route taken in \cite{AlonConn}: it uses neither expansion nor the local lemma. It has three steps.

\begin{enumerate}
\item[(i)] \emph{An algebraic base graph.} We label the edges of $K_{d,d}$ by vectors of $\FF_{2}^{d}$, setting $\lambda(x_{i}y_{j})=T^{i}e_{j}$ for an operator $T\in\GL_{d}(2)$ with irreducible characteristic polynomial. Irreducibility makes the $2^{d}-1$ selected subgraphs have invertible biadjacency matrices; what has to be arranged is that their supports are connected and contain a cycle (Section~\ref{sec:labels}).
\item[(ii)] \emph{A first-moment count over one conjugacy class.} For $d\ge 7$, choosing $T$ uniformly in a single irreducible conjugacy class of $\GL_{d}(2)$ makes each selected biadjacency matrix uniform among invertible matrices with prescribed first row; counting disconnected supports and spanning-tree supports and summing over all $2^{d}-1$ covectors gives a total failure probability below $1$ (Section~\ref{sec:prob}). Exact certificates handle $4\le d\le 8$ (Section~\ref{sec:certificates}).
\item[(iii)] \emph{Voltage lifts.} Random $\FF_{p}$-voltage covers of the labelled base graph preserve $d$-regularity, bipartiteness and simplicity, and keep every selected subgraph connected, provided each selected subgraph of the base contains a cycle. This is exactly why step (ii) certifies a cycle and not merely connectivity (Section~\ref{sec:lifts}).
\end{enumerate}

The computer enters only in step~(ii)'s finite complement: five explicit matrices are checked over all $15+31+63+127+255=491$ nonzero covectors, by a short program using exact integer arithmetic and the Python standard library, reproduced in Appendix~\ref{app:code}. Everything else is symbolic. Section~\ref{sec:concluding} lists what remains open.

\section{Vector labels on \texorpdfstring{$K_{d,d}$}{K(d,d)}}\label{sec:labels}

We first record the upper bound in the form we use.

\begin{lemma}\label{lem:upper}
If $H$ has a vertex of degree $d$ and at least two vertices, then $m(H)\le 2^{d}$.
\end{lemma}

\begin{proof}
Fix a vertex $v$ of degree $d$ and let $\delta_{H}(v)$ be the set of edges incident with it. The map $A\mapsto A\cap\delta_{H}(v)$ is injective on any connectivity code: if two codewords have the same restriction to $\delta_{H}(v)$, their symmetric difference isolates $v$ and is therefore not a connected spanning subgraph. There are $2^{d}$ possible restrictions.
\end{proof}

Let $W=\FF_{2}^{d}$ with dual $W^{*}$. Given a labelling $\lambda\colon E(H)\to W$ and $a\in W^{*}$ put
\[
E_{a}=\{e\in E(H): a(\lambda(e))=1\}.
\]
Since $a(\lambda(e))\ne b(\lambda(e))$ if and only if $(a+b)(\lambda(e))=1$, we have $E_{a}\triangle E_{b}=E_{a+b}$. The following is the labelling criterion of \cite[Lemma 2.1]{AlonConn}, in the equivalent formulation we shall use; recall that a set of edges meets every cut of $H$ if and only if it is a connected spanning subgraph.

\begin{lemma}\label{lem:criterion}
If $E_{a}$ is a connected spanning subgraph of $H$ for every nonzero $a\in W^{*}$, then $\{E_{a}: a\in W^{*}\}$ is a linear connectivity code for $H$ of dimension $d$, so $m(H)\ge 2^{d}$.
\end{lemma}

From now on $H=K_{d,d}$, with left vertices $x_{0},\dots,x_{d-1}$ and right vertices $y_{0},\dots,y_{d-1}$. Let $e_{0},\dots,e_{d-1}$ be the standard basis of $W$, let $T\in\GL_{d}(2)$, and set
\begin{equation}\label{eq:label}
\lambda(x_{i}y_{j})=T^{i}e_{j}\qquad (0\le i,j<d).
\end{equation}
For a nonzero row vector $a\in W^{*}$ the selected graph $E_{a}$ has biadjacency matrix
\begin{equation}\label{eq:biadj}
M_{a}(T)=\bigl(a\,T^{i}e_{j}\bigr)_{0\le i,j<d},
\end{equation}
whose rows are $a,\,aT,\,\dots,\,aT^{d-1}$. For a binary matrix $M$ let $\Gam(M)$ denote its bipartite support graph, with one left vertex per row, one right vertex per column, and an edge for each nonzero entry; thus $E_{a}$ and $\Gam(M_{a}(T))$ are the same graph.

\begin{lemma}\label{lem:irred}
If the characteristic polynomial of $T$ is irreducible of degree $d$, then $M_{a}(T)$ is invertible for every nonzero $a\in W^{*}$.
\end{lemma}

\begin{proof}
Irreducibility makes $\FF_{2}[T]\cong\FF_{2^{d}}$ a field, and $W$ a one-dimensional vector space over it; the same holds for the dual action on $W^{*}$. Hence every nonzero covector is cyclic, so $a,aT,\dots,aT^{d-1}$ is a basis of $W^{*}$.
\end{proof}

\begin{remark}\label{rem:stars}
Invertibility of $M_{a}(T)$ does not by itself imply that its support is connected, and it is connectivity that the next two sections must supply. Conversely, once every $E_{a}$ is a connected spanning subgraph, the labels on the $d$ edges at any single vertex automatically form a basis of $W$: otherwise some nonzero $a$ would annihilate all of them and isolate that vertex in $E_{a}$. So the one-vertex-cut condition of \cite[Lemma 2.1]{AlonConn} needs no separate verification here.
\end{remark}

\section{A probabilistic irreducible operator for \texorpdfstring{$d\ge 7$}{d>=7}}\label{sec:prob}

Throughout this section fix a monic irreducible polynomial of degree $d$ over $\FF_{2}$ and let $T$ be chosen uniformly at random from the conjugacy class $\mathcal{T}\subseteq\GL_{d}(2)$ consisting of the matrices with that characteristic polynomial. Write $g_{r}=|\GL_{r}(2)|$ with $g_{0}=1$, so that
\begin{equation}\label{eq:grec}
g_{r+1}=2^{r}\bigl(2^{r+1}-1\bigr)g_{r},\qquad g_{d}=\prod_{i=0}^{d-1}\bigl(2^{d}-2^{i}\bigr).
\end{equation}

\begin{lemma}[Uniform cyclic basis]\label{lem:uniform}
Fix $0\ne a\in W^{*}$. As $T$ ranges uniformly over $\mathcal{T}$, the matrix $M_{a}(T)$ is uniformly distributed over the set of all invertible $d\times d$ binary matrices whose first row is $a$.
\end{lemma}

\begin{proof}
The centraliser in $\GL_{d}(2)$ of an irreducible cyclic operator is $\FF_{2^{d}}^{\times}$, of order $2^{d}-1$, so $|\mathcal{T}|=g_{d}/(2^{d}-1)$. This is exactly the number of ordered bases $(v_{0},\dots,v_{d-1})$ of $W^{*}$ with $v_{0}=a$, namely $\prod_{i=1}^{d-1}(2^{d}-2^{i})$.

The map $T\mapsto(a,aT,\dots,aT^{d-1})$ sends $\mathcal{T}$ into that set by Lemma~\ref{lem:irred}, and it is injective: right multiplication by $T$ sends $v_{i}$ to $v_{i+1}$ for $i<d-1$, while $v_{d-1}T$ is determined by the prescribed characteristic polynomial through the Cayley--Hamilton relation, so the ordered basis determines $T$. Conversely, given an ordered basis starting with $a$, these rules define an operator whose matrix in that basis is the companion matrix of the fixed irreducible polynomial, hence an element of $\mathcal{T}$. The map is therefore a bijection between two sets of equal size, and the claim follows.
\end{proof}

We also need a count of spanning trees of $K_{d,d}$ by the neighbourhood of a fixed vertex.

\begin{lemma}\label{lem:trees}
Let $1\le w\le d$ and let $N$ be a set of $w$ right vertices of $K_{d,d}$. The number of spanning trees of $K_{d,d}$ in which the neighbourhood of $x_{0}$ is exactly $N$ equals $\tfrac{w}{d}(d-1)^{d-w}d^{\,d-1}$.
\end{lemma}

\begin{proof}
The bipartite Prüfer correspondence identifies the spanning trees of $K_{d,d}$ with pairs $(\alpha,\beta)$ of sequences of length $d-1$ over $\{0,\dots,d-1\}$, in such a way that $\deg(x_{i})=1+\#\{\text{occurrences of }i\text{ in }\alpha\}$. Hence the number of spanning trees with $\deg(x_{0})=w$ is $\binom{d-1}{w-1}(d-1)^{d-w}d^{\,d-1}$: choose the positions of $0$ in $\alpha$, fill the remaining $d-w$ positions with the other $d-1$ symbols, and take $\beta$ arbitrary. Relabelling the right vertices shows that all $\binom{d}{w}$ neighbourhoods of size $w$ occur equally often, and $\binom{d-1}{w-1}/\binom{d}{w}=w/d$. Summing over $w$ returns Cayley's count $d^{2d-2}$ for $K_{d,d}$, as it must.
\end{proof}

\begin{lemma}[Bad-support estimate]\label{lem:bad}
Let $d\ge 7$. With positive probability a uniformly random $T\in\mathcal{T}$ satisfies, simultaneously for every nonzero $a\in W^{*}$, that $\Gam(M_{a}(T))$ is connected and contains a cycle.
\end{lemma}

\begin{proof}
Fix a nonzero $a\in W^{*}$ of Hamming weight $w$. By Lemma~\ref{lem:uniform} the matrix $M=M_{a}(T)$ is uniform over the $g_{d}/(2^{d}-1)$ invertible matrices with first row $a$. Note first that $M$ has no zero row and no zero column, so $\Gam(M)$ has no isolated vertex; in particular $\Gam(M)$ is a spanning subgraph of $K_{d,d}$, and if it is connected and acyclic then it is a spanning tree.

\smallskip
\noindent\emph{Disconnected supports.} Suppose $\Gam(M)$ is disconnected and let the component containing the first row consist of $s$ rows and $t$ columns. The $s$ rows are supported on those $t$ columns, so invertibility forces $t\ge s$; applying this to every component and summing gives $t=s$ for each. The $s$ columns contain the $w$ prescribed neighbours of the first row, so $w\le s\le d-1$. Choosing the remaining $s-1$ rows and $s-w$ columns and then bounding the two diagonal blocks by arbitrary invertible blocks with the prescribed first row gives
\begin{equation}\label{eq:disc}
\Pr\bigl[\Gam(M)\ \text{disconnected}\bigr]\ \le\ \frac{2^{d}-1}{g_{d}}\sum_{s=w}^{d-1}\binom{d-1}{s-1}\binom{d-w}{s-w}\frac{g_{s}}{2^{s}-1}\,g_{d-s}.
\end{equation}
Summing \eqref{eq:disc} over the $\binom{d}{w}$ covectors of weight $w$ and over $w$, and using
\[
\sum_{w=1}^{s}\binom{d}{w}\binom{d-w}{s-w}=\binom{d}{s}\sum_{w=1}^{s}\binom{s}{w}=\binom{d}{s}\bigl(2^{s}-1\bigr),
\]
the factors $2^{s}-1$ cancel and the probability that at least one support is disconnected is at most
\begin{equation}\label{eq:Dd}
D_{d}\ :=\ \frac{2^{d}-1}{g_{d}}\sum_{s=1}^{d-1}\binom{d-1}{s-1}\binom{d}{s}g_{s}\,g_{d-s}.
\end{equation}

Pair the summands with indices $s=r$ and $s=d-r$. Writing
\[
R_{d,r}=\frac{(2^{d}-1)g_{r}g_{d-r}}{g_{d}},\qquad B_{d,r}=\binom{d}{r}^{2}R_{d,r},
\]
and using $R_{d,r}=R_{d,d-r}$ together with $\binom{d-1}{r-1}+\binom{d-1}{r}=\binom{d}{r}$, the paired contribution is exactly $B_{d,r}$ for $r<d/2$; when $d$ is even the middle summand equals $\tfrac12 B_{d,d/2}$, since $\binom{d-1}{d/2-1}=\tfrac12\binom{d}{d/2}$. By \eqref{eq:grec},
\[
B_{d,1}=\frac{d^{2}}{2^{d-1}},\qquad
\frac{B_{d,r+1}}{B_{d,r}}=\Bigl(\frac{d-r}{r+1}\Bigr)^{2}2^{\,2r-d+1}\,\frac{2^{r+1}-1}{2^{d-r}-1}.
\]

We claim each successive actual contribution is smaller than the previous one by a factor less than $1/5$. Put $k=d-2r$. If the transition is into the halved middle term then $d=2q$ and $r=q-1$, and since $d\ge 7$ we have $q\ge 4$; the extra factor $\tfrac12$ gives
\[
\frac{\tfrac12 B_{2q,q}}{B_{2q,q-1}}<\frac{1}{8}\Bigl(1+\frac{1}{q}\Bigr)^{2}\le\frac{25}{128}<\frac{1}{5}.
\]
Every other transition has $k\ge 3$, hence $r+k=d-r\ge 4$ and $\frac{2^{r+1}-1}{2^{r+k}-1}<\frac{16}{15}2^{1-k}$, so that
\[
\frac{B_{d,r+1}}{B_{d,r}}<\frac{16}{15}\Bigl(\frac{r+k}{r+1}\Bigr)^{2}2^{2-2k}.
\]
If $k=3$ then $d\ge 7$ forces $r\ge 2$ and the right-hand side is less than $5/27$; if $k=4$ then $r\ge 2$ and it is less than $1/15$; if $k\ge 5$ then $(r+k)/(r+1)\le k$ and it is at most $\tfrac{16}{15}k^{2}2^{2-2k}\le 5/48$. In all cases the factor is below $1/5$, whence
\begin{equation}\label{eq:Ddbound}
D_{d}\ \le\ B_{d,1}\sum_{j\ge 0}5^{-j}\ =\ \frac{5}{4}\cdot\frac{d^{2}}{2^{d-1}}.
\end{equation}

\smallskip
\noindent\emph{Acyclic supports.} If $\Gam(M)$ is connected but acyclic it is a spanning tree of $K_{d,d}$, and then $M$ is determined by its support. By Lemma~\ref{lem:trees} the number of such $M$ with first row $a$ of weight $w$ is $\tfrac{w}{d}(d-1)^{d-w}d^{\,d-1}$, and
\[
\sum_{w=1}^{d}\binom{d}{w}\frac{w}{d}(d-1)^{d-w}=\sum_{w=1}^{d}\binom{d-1}{w-1}(d-1)^{d-w}=d^{\,d-1}.
\]
Hence the probability that some support is a spanning tree is at most
\begin{equation}\label{eq:Td}
T_{d}\ :=\ \frac{2^{d}-1}{g_{d}}\,d^{\,2d-2}.
\end{equation}
Since $g_{d}=2^{d^{2}}\prod_{i=1}^{d}(1-2^{-i})>\tfrac14 2^{d^{2}}$, we get $T_{d}<4\,d^{\,2d-2}2^{\,d-d^{2}}$.

\smallskip
\noindent\emph{Conclusion.} Both $\tfrac54 d^{2}2^{1-d}$ and $4d^{2d-2}2^{d-d^{2}}$ are decreasing for $d\ge 7$, and at $d=7$
\[
D_{7}+T_{7}\ \le\ \frac{245}{256}+\frac{4\cdot 7^{12}}{2^{42}}\ <\ 0.970\ <\ 1 .
\]
So for every $d\ge 7$ the union bound leaves positive probability that no nonzero covector has a disconnected or acyclic support.
\end{proof}

\begin{remark}\label{rem:margin}
The estimate \eqref{eq:Ddbound} is deliberately crude, and the margin at $d=7$ is larger than it looks: evaluating \eqref{eq:Dd} and \eqref{eq:Td} exactly gives
\[
D_{7}+T_{7}=0.80008\ldots,\qquad D_{8}+T_{8}=0.50503\ldots,\qquad D_{9}+T_{9}=0.31736\ldots,
\]
so no delicacy in the constants is being exploited. The threshold is nevertheless intrinsic to this first-moment argument rather than an artefact of the estimates: the exact value of $D_{6}+T_{6}$ is $1.41666\ldots>1$, so degrees $4,5,6$ genuinely require the certificates of the next section.
\end{remark}

\section{Exact certificates for \texorpdfstring{$4\le d\le 8$}{4<=d<=8}}\label{sec:certificates}

Lemma~\ref{lem:bad} leaves $d\in\{4,5,6\}$. We record verified matrices for these three degrees and, for redundancy, also for $d=7$ and $d=8$; the probabilistic argument is then needed only for $d\ge 9$, where the exact failure bound is below $0.32$ by Remark~\ref{rem:margin}. All entries are in $\FF_{2}$.
\[
T_{4}=\begin{pmatrix}
1&0&0&1\\ 1&1&1&0\\ 1&0&1&0\\ 0&1&0&1
\end{pmatrix},\qquad
T_{5}=\begin{pmatrix}
0&1&0&1&1\\ 0&0&1&0&0\\ 1&1&0&0&0\\ 1&0&1&0&0\\ 0&0&0&1&0
\end{pmatrix},\qquad
T_{6}=\begin{pmatrix}
0&1&0&0&0&0\\ 0&1&0&0&1&0\\ 1&0&0&0&0&1\\ 0&0&1&0&0&1\\ 0&0&1&0&0&0\\ 0&0&0&1&0&1
\end{pmatrix},
\]
\[
T_{7}=\begin{pmatrix}
0&1&0&0&1&1&0\\ 0&0&0&1&1&1&0\\ 1&0&0&1&1&0&0\\ 0&0&1&0&0&1&1\\ 0&0&0&0&0&1&0\\ 0&0&1&0&1&0&0\\ 0&0&1&0&0&0&0
\end{pmatrix},\qquad
T_{8}=\begin{pmatrix}
1&1&0&0&0&0&1&1\\ 1&0&1&0&1&1&1&1\\ 0&0&0&0&0&1&0&1\\ 0&0&0&0&1&0&0&1\\ 1&0&0&1&0&0&0&0\\ 1&0&0&0&0&0&0&0\\ 0&0&0&1&0&0&1&0\\ 0&1&0&0&1&1&0&0
\end{pmatrix}.
\]
Their characteristic polynomials are, respectively,
\[
x^{4}+x+1,\quad x^{5}+x^{2}+1,\quad x^{6}+x+1,\quad x^{7}+x+1,\quad x^{8}+x^{4}+x^{3}+x+1,
\]
all irreducible over $\FF_{2}$. Exact enumeration of all nonzero covectors $a$ gives the following certificate summary, where $\beta_{1}=|E|-|V|+1$ since all supports turn out to be connected.

\begin{center}
\begin{tabular}{cccccc}
\toprule
$d$ & $\#\{a\ne 0\}$ & $\min_{a}|E(\Gam(M_{a}))|$ & $\min_{a}\beta_{1}(\Gam(M_{a}))$ & \# disconnected & \# trees\\
\midrule
$4$ & $15$  & $8$  & $1$ & $0$ & $0$\\
$5$ & $31$  & $10$ & $1$ & $0$ & $0$\\
$6$ & $63$  & $12$ & $1$ & $0$ & $0$\\
$7$ & $127$ & $14$ & $1$ & $0$ & $0$\\
$8$ & $255$ & $23$ & $8$ & $0$ & $0$\\
\bottomrule
\end{tabular}
\end{center}

The verifier is reproduced in Appendix~\ref{app:code}. It uses only exact integer arithmetic and the Python standard library, recomputes each characteristic polynomial and its irreducibility by trial division, and independently checks all $491$ covectors; it runs in a fraction of a second.

\begin{remark}\label{rem:companion}
The choice of matrix, and not merely of characteristic polynomial, is essential, so the certificates above cannot be replaced by the obvious canonical choice. If $T$ is the companion matrix of an irreducible $f$ and $W$ is identified with $\FF_{2}[x]/(f)$ with its power basis, then \eqref{eq:label} becomes $\lambda(x_{i}y_{j})=\alpha^{i+j}$ for $\alpha$ a root of $f$, so $M_{a}(T)$ is the Hankel matrix $\bigl(\mathrm{Tr}(c\alpha^{i+j})\bigr)_{i,j}$ for a suitable $c\ne 0$. These matrices are always invertible, but their supports are not always connected: taking the lexicographically first irreducible polynomial of each degree, the companion matrix fails for every $d$ with $4\le d\le 12$, with between $5$ and $45$ disconnected supports in each case. By contrast, a random conjugate of the companion matrix succeeds with substantial probability already at $d=7$ (about $84\%$ over $50$ trials), which is what Lemma~\ref{lem:bad} makes precise.
\end{remark}

Combining Lemma~\ref{lem:bad} with the table gives the following.

\begin{proposition}\label{prop:base}
For every $d\ge 4$ there is an irreducible $T\in\GL_{d}(2)$ such that the labelling \eqref{eq:label} of $K_{d,d}$ has the property that $E_{a}$ is a connected spanning subgraph containing a cycle, for every nonzero $a\in W^{*}$.
\end{proposition}

With Lemma~\ref{lem:criterion} and Lemma~\ref{lem:upper}, Proposition~\ref{prop:base} already yields Corollary~\ref{cor:kdd}: $m(K_{d,d})=2^{d}$ for every $d\ge 4$.

\section{Cyclic lifts}\label{sec:lifts}

It remains to produce infinitely many examples in each degree. We use the standard voltage-graph construction; see \cite{GrossTucker} for background. Given a graph $B$ with an arbitrary orientation of its edges and a map $\varphi$ from the oriented edges to a finite group $\Gamma$ (with reversed edges receiving inverse values), the \emph{derived cover} $B^{\varphi}$ has vertex set $V(B)\times\Gamma$ and an edge from $(u,g)$ to $(v,g\varphi(uv))$ for each oriented edge $uv$ of $B$ and each $g\in\Gamma$.

\begin{lemma}[Simultaneously connected cyclic lifts]\label{lem:lift}
Let $B$ be a finite graph with an edge labelling in $\FF_{2}^{d}$ such that every selected graph $B_{a}$, $a\ne 0$, is connected, spanning and contains a cycle. Then there are infinitely many finite covers $\widetilde{B}$ of $B$ in which every lifted selected graph $\widetilde{B}_{a}$ is connected and spanning. If $B$ is finite, simple, bipartite and $d$-regular, the covers may be chosen with all of these properties.
\end{lemma}

\begin{proof}
Fix a prime $p>2^{d}-1$, orient the edges of $B$ arbitrarily, assign to each oriented edge an independent uniform voltage in $\FF_{p}$, and form the derived cover. Pull the labels back to all lifted edges, so that $\widetilde{B}_{a}$ is precisely the lift of $B_{a}$.

Fix $a\ne 0$ and a spanning tree of $B_{a}$, which exists since $B_{a}$ is connected and spanning. Conditionally on the voltages of the tree edges, the voltages of the $\beta_{1}(B_{a})$ fundamental cycles are independent and uniform in $\FF_{p}$, because each such cycle contains its own non-tree edge with coefficient $\pm1$. The lift of $B_{a}$ is connected if and only if these cycle voltages generate $\FF_{p}$; as $p$ is prime and $\beta_{1}(B_{a})\ge 1$ by hypothesis, failure has probability $p^{-\beta_{1}(B_{a})}\le p^{-1}$. A union bound over the at most $2^{d}-1$ nonzero covectors gives a positive probability that all lifts are simultaneously connected, since $(2^{d}-1)/p<1$.

A derived cover of a finite simple bipartite $d$-regular graph is again finite, simple, bipartite and $d$-regular: no loops arise because $B$ has none, and no multiple edges arise because distinct lifted edges joining the same pair of vertices would require parallel edges in $B$. Finally, there are infinitely many eligible primes.
\end{proof}

\begin{proof}[Proof of Theorem~\ref{thm:main}]
Let $d\ge 4$. Take the labelled base graph $B=K_{d,d}$ supplied by Proposition~\ref{prop:base} and a good voltage lift $\widetilde{B}$ given by Lemma~\ref{lem:lift} for a prime $p>2^{d}-1$. For each nonzero $a$, the selected graph in $\widetilde{B}$ is the lift of $B_{a}$ and is therefore connected and spanning, so by Lemma~\ref{lem:criterion} the $2^{d}$ sets $\widetilde{B}_{a}$ form a linear connectivity code and $m(\widetilde{B})\ge 2^{d}$. Since $\widetilde{B}$ is $d$-regular, Lemma~\ref{lem:upper} gives the matching upper bound, so $m(\widetilde{B})=2^{d}$. Taking one cover for each eligible prime produces graphs of $2dp$ vertices, hence infinitely many pairwise nonisomorphic examples, and $f(d)=2^{d}$.
\end{proof}

\begin{remark}
The hypothesis that each $B_{a}$ contains a cycle is essential and not merely convenient: every voltage lift of a tree is a disjoint union of copies of that tree, so a selected graph that is a spanning tree of the base can never lift to a connected graph. The tree count \eqref{eq:Td} in Lemma~\ref{lem:bad} was set up to eliminate exactly this obstruction, and the certificates of Section~\ref{sec:certificates} record $\min_{a}\beta_{1}\ge 1$ for the same reason.
\end{remark}

\section{Concluding remarks}\label{sec:concluding}

Corollary~\ref{cor:complete} determines $f(d)$ for every $d$, but several natural questions in the neighbourhood of Question~\ref{q:alon} remain open.

\begin{enumerate}
\item[(a)] \emph{Small-degree expanders.} Alon asks in \cite{AlonConn} whether $m(H)=2^{d}$ for every $d$-regular Ramanujan graph with $d\ge 4$. Our examples are cyclic lifts of $K_{d,d}$ and are far from Ramanujan, so this remains open; more generally, it would be interesting to determine the least $d_{0}$ for which the expander criterion of \cite{AlonConn} can be made to work, which is a different question from Question~\ref{q:alon} and is untouched here.
\item[(b)] \emph{Non-bipartite examples.} The construction produces bipartite graphs in every degree $d\ge 4$. Is there an equally uniform non-bipartite construction? A natural candidate would replace $K_{d,d}$ by a Cayley graph of $\FF_{2}^{d}$ with the labelling given by the connection set itself.
\item[(c)] \emph{Linear versus non-linear.} All the extremal codes known in this setting, including ours, are linear. Alon \cite{AlonConn} asks whether there are graphs $H$ for which non-linear connectivity codes beat linear ones; our construction sheds no light on this.
\item[(d)] \emph{Products of cycles.} Determining $m(C_{m}\times C_{n})$ for all $m,n$ remains open \cite{AlonConn,BGMW}, as does the complexity of computing or approximating $m(H)$ for a given input graph.
\end{enumerate}

\subsection*{Acknowledgements}
The computations in Section~\ref{sec:certificates} and Remarks~\ref{rem:margin} and~\ref{rem:companion} were carried out with the program in Appendix~\ref{app:code}.

\appendix

\section{Verifier}\label{app:code}

The following self-contained program recomputes, for each of the five matrices of Section~\ref{sec:certificates}, the characteristic polynomial over $\FF_{2}$ and its irreducibility, and then enumerates all nonzero covectors $a$, reporting the number of disconnected supports, the number of spanning-tree supports, and the minima of $|E|$ and $\beta_{1}$. All arithmetic is exact integer arithmetic; no external libraries are used.

\begin{lstlisting}[language=Python]
from itertools import permutations, product

def pmul(p, q):                      # product in F_2[x], polys as bitmasks
    r = 0
    while q:
        if q & 1: r ^= p
        p <<= 1; q >>= 1
    return r

def charpoly(T):                     # det(xI + T) over F_2, Leibniz formula
    d = len(T)
    M = [[T[i][j] ^ (2 if i == j else 0) for j in range(d)] for i in range(d)]
    tot = 0
    for perm in permutations(range(d)):
        t = 1
        for i in range(d):
            t = pmul(t, M[i][perm[i]])
            if t == 0: break
        tot ^= t
    return tot

def irreducible(p):                  # trial division by all polys of degree <= deg(p)/2
    d = p.bit_length() - 1
    if d < 1 or p & 1 == 0: return False
    for dd in range(1, d // 2 + 1):
        for c in range(1 << dd, 1 << (dd + 1)):
            r = p
            while r.bit_length() >= c.bit_length():
                r ^= c << (r.bit_length() - c.bit_length())
            if r == 0: return False
    return True

class DSU:
    def __init__(s, n): s.p = list(range(n))
    def find(s, x):
        while s.p[x] != x: s.p[x] = s.p[s.p[x]]; x = s.p[x]
        return x
    def union(s, x, y):
        x, y = s.find(x), s.find(y)
        if x != y: s.p[x] = y

def analyze(T):
    d = len(T)
    minE, minB, disc, tree = 10**9, 10**9, 0, 0
    for a in product([0, 1], repeat=d):
        if not any(a): continue
        rows, cur = [], a
        for _ in range(d):
            rows.append(cur)
            cur = tuple(sum(cur[i] * T[i][j] for i in range(d)) % 2 for j in range(d))
        ds, E = DSU(2 * d), 0
        for i in range(d):
            for j in range(d):
                if rows[i][j]: ds.union(i, d + j); E += 1
        comps = len({ds.find(v) for v in range(2 * d)})
        if comps > 1: disc += 1
        b1 = E - 2 * d + comps
        if comps == 1 and b1 == 0: tree += 1
        minE, minB = min(minE, E), min(minB, b1)
    return minE, minB, disc, tree

MATRICES = {
 4: [[1,0,0,1],[1,1,1,0],[1,0,1,0],[0,1,0,1]],
 5: [[0,1,0,1,1],[0,0,1,0,0],[1,1,0,0,0],[1,0,1,0,0],[0,0,0,1,0]],
 6: [[0,1,0,0,0,0],[0,1,0,0,1,0],[1,0,0,0,0,1],[0,0,1,0,0,1],
     [0,0,1,0,0,0],[0,0,0,1,0,1]],
 7: [[0,1,0,0,1,1,0],[0,0,0,1,1,1,0],[1,0,0,1,1,0,0],[0,0,1,0,0,1,1],
     [0,0,0,0,0,1,0],[0,0,1,0,1,0,0],[0,0,1,0,0,0,0]],
 8: [[1,1,0,0,0,0,1,1],[1,0,1,0,1,1,1,1],[0,0,0,0,0,1,0,1],[0,0,0,0,1,0,0,1],
     [1,0,0,1,0,0,0,0],[1,0,0,0,0,0,0,0],[0,0,0,1,0,0,1,0],[0,1,0,0,1,1,0,0]],
}

for d, T in sorted(MATRICES.items()):
    cp = charpoly(T)
    assert cp.bit_length() - 1 == d and irreducible(cp)
    minE, minB, disc, tree = analyze(T)
    assert disc == 0 and tree == 0 and minB >= 1
    print(f"d={d}  charpoly={cp:b}  #a={2**d-1}  min|E|={minE}  "
          f"min beta1={minB}  disconnected={disc}  trees={tree}")
\end{lstlisting}

\end{document}